\documentclass[12pt]{amsart}
\usepackage{amssymb}
\usepackage{epsfig}
\usepackage{graphicx}
\numberwithin{equation}{section}

\input xy
\xyoption{all}

\calclayout
\allowdisplaybreaks[3]

\theoremstyle{plain}
\newtheorem{prop}{Proposition}

\newtheorem{theo}[prop]{Theorem}
\newtheorem{coro}[prop]{Corollary}

\newtheorem{lemm}[prop]{Lemma}
\newtheorem{assu}[prop]{Assumption}

\theoremstyle{definition}
\newtheorem{defi}[prop]{Definition}

\newtheorem{rema}[prop]{Remark}

\newcommand{\ovl}{\overline}

\def\G{{\mathcal G}}

\def\Hom{{\rm Hom}}

\def\Gal{{\rm  G}}
\def\G{{\mathcal G}}   

\newcommand{\Ker}{{\rm Ker}}

\def\ra{\rightarrow}

\def\P{{\mathbb P}}
\def\Q{{\mathbb Q}}

\def\Z{{\mathbb Z}}

\def\N{{\mathbb N}}

\def\Div{{\rm Div}}

\def\ra{\rightarrow}

\def\ovl{\overline}

\def\Gal{{\rm Gal}}

\def\ra{\rightarrow}

\def\P{{\mathbb P}}
\def\Q{{\mathbb Q}}

\def\Z{{\mathbb Z}}

\def\N{{\mathbb N}}

\def\rK{{\rm K}}

\def\Div{{\rm Div}}

\makeatother
\makeatletter

\author{Fedor Bogomolov}
\address{Courant Institute\\
                New York University \\
                New York, NY 10012 \\
                USA }
\email{bogomolo@cims.nyu.edu}

\author{Yuri Tschinkel}
\address{Courant Institute\\
                New York University \\
                New York, NY 10012 \\
                USA }
\email{tschinkel@cims.nyu.edu}

\title[Milnor $\rK_2$]{Milnor $\rK_2$ and field homomorphisms}

\begin{document}
\date{\today}

\begin{abstract}
We prove that the function field of an 
algebraic variety of dimension $\ge 2$ over 
an algebraically closed field 
is completely determined by its 
first and second Milnor $\rK$-groups.
\end{abstract}

\maketitle
\tableofcontents

\section{Introduction}
\label{sect:introduction}

In this paper we study the problem of reconstruction of 
field homomorphisms from group-theoretic data. A prototypical 
example is the reconstruction of function fields of algebraic varieties
from their absolute Galois group, a central problem in 
``anabelian geometry'' (see \cite{tamagawa}, \cite{mochizuki}, \cite{mochizuki2}, \cite{pop}).  
Within this theory, an important question is the ``section conjecture'', i.e., 
the problem of detecting homomorphisms of fields on the level of homomorphisms 
of their Galois groups. In the language of algebraic geometry, one is 
interested in obstructions to the existence of points of algebraic varieties
over higher-dimensional function fields, or equivalently, rational sections 
of fibrations. Here we study group theoretic objects which are dual, in some
sense, to small pieces of the Galois group, obtained from 
the {\em abelianization} of the absolute Galois group and its canonical central extension. 
This connection will be explained in Section~\ref{sect:back}.

We now formulate the main results. 
In this paper, we work in characteristic zero. 
An element of an abelian group is called {\rm  primitive}, 
if it cannot be written as a nontrivial multiple in this group.

\begin{defi}
\label{defn:field}
Let $k$ be an infinite field. 
A field $K$ will be called {\em geometric} over $k$ if 
\begin{enumerate}
\item $k\subset K$;
\item for each $f\in K^*\setminus k^*$, the set $\{ f+\kappa\}_{\kappa \in k}$
has at most finitely many elements whose image
in $K^*/k^*$ is nonprimitive. 
\end{enumerate}
\end{defi}

If $X$ is an algebraic variety over 
an algebraically closed field $k$
of characteristic zero then 
its function field $K=k(X)$ is geometric over $k$. 
There exist other examples, e.g., 
some {\em infinite} algebraic extensions of $k(X)$ are also
geometric over $k$. 

\begin{theo}
\label{thm:zero}
Let $K$, resp. $L$, be a geometric field
of transcendence degree $\ge 2$ over an
algebraically closed field $k$, resp. $l$, of characteristic zero.  
Assume that there exists an injective 
homomorphism of abelian groups
$$
\psi_1\,:\, K^*/k^*\ra L^*/l^*
$$
such that
\begin{enumerate}
\item the image of $\psi_1$ contains one primitive element in $L^*/l^*$ and
two elements whose lifts to $L^*$ are algebraically independent over $l$;
\item for each $f\in K^* \setminus k^*$
there exists a $g\in L$ such that 
$$
\psi_1\left(\overline{k(f)}^*/k^*\cap K^*/k^*\right) 
\subseteq \overline{l(g)}^*/l^*\cap L/l^*.
$$ 
\end{enumerate}
Then there exists a field embedding
$$
\psi\,:\, K\ra L
$$
which induces either $\psi_1$ or $\psi_1^{-1}$. 
\end{theo}

\begin{rema}
An analogous statement holds in positive characteristic.
The final steps of the proof in Section~\ref{sect:proof}    
are more technical due to the presence of $p^n$-powers of ``projective lines''. 
\end{rema}

Let $K$ be a field. 
Denote by $\rK^{M}_i(K)$ the $i$-th
Milnor K-group of $K$. Recall that 
$$
\rK^M_1(K)=K^*
$$ 
and that there is a canonical surjective homomorphism
$$
\sigma_K\,:\, \rK^M_1(K)\otimes \rK^M_1(K)\ra \rK_2^M(K)
$$
whose kernel is generated by  $x\otimes (1-x)$, for $x\in K^*\setminus 1$
(see \cite{milnor} for more background on K-theory).  
Put 
$$
\bar{\rK}^M_i(K):= \rK^M_i(K)/{\rm infinitely  \,\, divisible}, \quad i=1,2.
$$
The homomorphism $\sigma_K$ is compatible with reduction 
modulo infinitely divisible elements. 
As an application of Theorem~\ref{thm:zero} we prove the following result. 

\begin{theo}
\label{thm:main}
Let $K$ and $L$ be function fields of algebraic varieties of dimension $\ge 2$ 
over an algebraically closed field $k$, resp. $l$.  
Let
\begin{equation}
\label{eqn:psi*}
\psi_1\,:\, \bar{\rK}^M_1(K)\ra \bar{\rK}^M_1(L)
\end{equation}
be an injective homomorphism of abelian groups
such that the following diagram of abelian group homomorphisms is commutative

\

\centerline{
\xymatrix{
{\bar{\rK}^M_1(K)}\otimes {\bar{\rK}^M_1(K)} \ar[rr]^{\psi_1\otimes \psi_1} 
\ar[d]_{\sigma_K}  && 
\ar[d]^{\sigma_L} \bar{\rK}^M_1(L)\otimes \bar{\rK}^M_1(L) \\
    \bar{\rK}^M_2(K)              \ar[rr]_{\psi_2} && \bar{\rK}^M_2(L). 
}
}

\

\noindent
Assume further that $\psi_1(K^*/k^*)$ is not contained in $E^*/k^*$ for
any 1-dimensional subfield $E\subset L$. 
Then there exist a homomorphism of fields 
$$
\psi\,:\, K\ra L,
$$
and an $r\in \Q$ such that the induced map on $K^*/k^*$ coincides with 
the $r$-th power of $\psi_1$.
\end{theo}

\

In particular, the assumptions are satisfied when $\psi_1$ is an isomorphism
of abelian groups.  
In this case, Theorem~\ref{thm:main} states that a function field of 
transcendence degree $\ge 2$ over an algebraically closed ground field 
of characteristic zero
is determined by its first and second Milnor K-groups.

\

\noindent
{\bf Acknowledgments:}
The first author was partially supported by NSF grant DMS-0701578. 
He would like to thank the Clay Mathematics Institute for financial
support and Centro Ennio De Giorgi in Pisa for hospitality 
during the completion of the manuscript.
The second author was partially supported by
NSF grant DMS-0602333.

We are grateful to B. Hassett, M. Rovinsky and Yu. Zarhin
for their interest and useful suggestions.

\section{Background}
\label{sect:back}

The problem considered in this paper has the appearance of 
an abstract algebraic question. However, it is intrinsically 
related to our program to develop a 
skew-symmetric version of the theory of fields, and especially, 
function fields of algebraic varieties.

Let $K$ be a field and $\G_K$ its absolute
Galois group, i.e., the Galois group of
a maximal separable extension of $K$.  
It is a compact profinite group. 
We are interested in the quotient
$$
\G^c_K = \G_K/[\G_K, [\G_K,\G_K]]
$$
and its maximal topological pro-$\ell$-completion
$$
\G_{K,\ell}^c, \quad \ell \neq {\rm char}(K).
$$
The group $\G_{K,\ell}^c$ is a central pro-$\ell$-extension of the
pro-$\ell$-completion of the
abelianization $\G_K^a$ of $\G_K$.

We now assume that $K$ is the function field of an algebraic variety over 
an algebraically closed
ground field $k$. In this case, 
$\G^a_{K,\ell}$
is a torsion-free topological pro-$\ell$-group which
is dual to the torsion-free abelian group $K^*/k^*$, i.e.,
there is a canonical identification
$$
\G^a_{K,\ell} = \Hom(K^*/k^*,\Z_{\ell}(1)), 
$$ 
via Kummer theory.
The group $\G^c_{K,\ell}$ 
admits a simple description in terms
of one-dimensional subfields of $K$, i.e., subfields of transcendence degree 1 
over $k$.
For each such subfield $E\subset K$, which is normally closed in $K$, we have
a surjective homomorphism $\G^c_{K,\ell} \ra \G^c_{E,\ell}$, 
where the image
is a {\em free} central pro-$\ell$-extension of the group
$\G^a_{E,\ell}$. 

Our main goal is to establish a functorial correspondence
between function fields of algebraic varieties $K$ and $L$, 
over algebraically closed ground fields $k$ and $l$, respectively,  
and corresponding topological groups $\G^c_K$, resp. $\G^c_{K,\ell}$.
We are aiming at a (conjectural) equivalence
between homomorphisms of function fields 
$$
\bar{\Psi} \,:\, K\to L
$$ 
and homomorphisms of topological groups
$$
\Psi^c_{\ell} \,  : \, \G^c_{K,\ell} \to \G^c_{L,\ell}.
$$ 
It is clear that $\bar{\Psi}$ induces (but not uniquely) a homomorphism 
$\Psi^c_{\ell}$ as above. The problem is to find conditions on
$\Psi^c_{\ell}$ such that it corresponds to some $\bar{\Psi}$.
In particular, $\Psi^c_{\ell}$ would give rise to homomorphisms of the full Galois groups
$\G_K\ra \G_L$. 

\begin{rema}
By a theorem of Stallings \cite{stallings}, 
a group homomorphism that induces an isomorphism on $H_1( - , \Z)$ and an
epimorphism on $H_2(-, \Z)$ induces an isomorphism on the lower central series. 

Thus we expect that $\G_{K,\ell}$ is in some sense the maximal pro-$\ell$-group 
with given $H_1$ and $H_2$. 
\end{rema}

Consider the diagram

\

\centerline{
\xymatrix{
  \G_{K,\ell}^c \ar[r] \ar[d] & \G_{L,\ell}^c  \ar[d] \\
  \G_{K,\ell}^a \ar[r] & \G_{L,\ell}^a
}
}

\

\noindent
The group  $\G^c_{K,\ell}$ can be identified
with a closed subgroup in the direct product 
of  free central pro-$\ell$-extensions 
$$
\prod_E \G^c_{E,\ell},
$$ 
where the product runs over
all normally closed one-dimensional subfields $E$ of $K$.
The homomorphisms  $\G^c_{K,\ell} \to\G^c_{E,\ell}$ are induced
from certain homomorphisms of abelian quotients $\G^a_{K,\ell} \to \G^a_{L,\ell}$,
namely those which commute with surjective maps of 
$\G^a_{K,\ell}$ and $\G^a_{L,\ell}$ 
to the abelian groups of one-dimensional subfields
of $K$ and $L$, respectively.

It is shown in \cite{bt} that in the case of functional
fields of transcendence degree 2 over $k=\bar{\mathbb F}_p$ and $\ell\neq p$,  
any isomorphism $\Psi^c_{\ell}$ defines an isomorphism between $K$ 
and some finite purely inseparable extension of $L$. In this paper 
we treat the first problem which
arises when we try to extend the result to general homomorphisms.
By the description above, it suffices 
to treat the corresponding homomorphisms 
of abelian groups 
$$
\Psi^a_{\ell}\,:\, \G^a_{K,\ell}\ra \G^a_{L,\ell}.
$$
By Kummer theory, these can be identified with  homomorphisms
$$
\Psi^*_{\ell} \,:\,\Hom(K^*/k^*,\Z_\ell)\to \Hom(L^*/l^*,\Z_\ell).
$$
The condition that 
$\Psi^c_{\ell}$ commutes with projections 
onto Galois groups of one-dimensional 
fields is the same as commuting with
projections 
$$
\Hom(K^*/k^*,\Z_\ell(1))\to \Hom(E^*,\Z_\ell(1)).
$$ 
If it were possible to dualize the picture we would  have a homomorphism
$$
\Psi^* : L^*/l^* \to K^*/k^*,
$$ 
mapping multiplicative groups of  
one-dimensional subfields in $L$ to multiplicative groups of 
one-dimensional subfields of $K$. 
This is the problem that we consider in the paper.

In order to solve the problem for Galois groups we need to consider
the maps  
$$
\hat{\Psi}^*_{\ell} \,:\,  \hat{L}^*\to \hat{K}^*,
$$ 
between $\ell$-completions
of the dual spaces (as in \cite{bt}) and to find conditions 
which would allow to reconstruct 
$\Psi^*$ from $\hat\Psi^*_{\ell}$.
This problem will be addressed in a future publication.


\section{Functional equations}
\label{sect:funct}
\begin{lemm}
\label{lemm:fg}
Let $x,y\in K$ be algebraically independent elements and
$z\in k(x,y)$ a nonconstant rational function.
Let $f,g\in k(t)^*$ be nonconstant functions such that 
$f(x)/g(y)\in k(z)$. 
Then there exist $\tilde{f},\tilde{g}\in k(t)^*$ such that
$$
k(z)=k(\tilde{f}(x)/\tilde{g}(y)).
$$
\end{lemm}

\begin{proof}
Write $z=p(x,y)/q(x,y)$, with coprime $p,q\in k[x,y]$. Then
$$
f(x)/g(y) = \prod_i (p/q -c_i)^{n_i}= q^{-\sum_i n_i} \prod_i (p - c_i q)^{n_i},
$$
modulo $k^*$, for pairwise distinct $c_i\in k$ and some $n_i\in \Z$.
The factors on the right are pairwise coprime, i.e., 
their divisors have no common components.
Thus the divisors of $q(x,y)$ and $p(x,y)- c_i q(x,y)$ are either
``vertical'' or ``horizontal'', i.e., 
$$
q(x,y)= t(x)u(y) \quad \text{ and } \quad p(x,y)- c_i q(x,y) = v_i(x)w_i(y),
$$ 
for some $t,u,v_i,w_i\in k(t)$. 
It follows that 
$$
z(x,y) - c_i = v_i(x)w_i(y)/t(x)u(y)
$$ 
and we can put
$\tilde{g}=v_i(x)/t(x)$ and $\tilde{f}=z(y)/w_i(y)$. 
\end{proof}

A rational function $f\in k(x,y)^*$ is called {\em homogeneous of degree $r$} if 
\begin{equation}
\label{eqn:lambdaa}
\lambda^r f(x,y)= f(\lambda x,\lambda y),\quad \text{ for all } \,\,\lambda\in k^*.
\end{equation}
A function $f$ is homogeneous of degree 0 iff $f\in k(x/y)^*$.

\begin{lemm}
\label{lemm:homo}
Let $p_1,p_2\in k(x,y)^*$ be rational functions with disjoint divisors. Assume that
$p_1(x,y)\cdot p_2(x,y)$ is homogeneous of degree $r$. Then $p_1$ is homogeneous of
degree $r_1$, $p_2$ is homogeneous of degree $r_2$ and $r_1+r_2=r$. 
\end{lemm}

\begin{coro}
\label{coro:homog}
Let $f,g\in k[t]$ be nonzero polynomials. Assume that
$p(x,y):=g(x)f(y)$ is homogeneous of degree $d\in \N$. 
Then 
\begin{align*}
g(x) & =ax^n\\
f(y) & =by^{d-n},
\end{align*}
for some $n\in \N$ and $a,b\in k^*$.  
\end{coro}

\begin{lemm}
\label{lemm:red}
Let $f,g\in k[t]$ be polynomials such that 
\begin{equation}
\label{eqn:red}
p(x,y)=ax^rf(y)-cy^rg(x)\in k[x,y]
\end{equation}
is homogeneous of degree $r\in \N$.
Then 
\begin{align*}
g(x)& =a_dx^r+a_0,\\
f(y)& =c_dy^r+c_0,
\end{align*}
and $a c_d-c a_d=0$.  
\end{lemm}

\begin{proof}
Write $g(x)=\sum_i a_ix^i$ and $f(y)=\sum_{j} c_jy^j$, substitute into 
the equation \eqref{eqn:red},
and use homogeneity. 
\end{proof}

\begin{lemm}
\label{lemm:reduction}
Let $f_1,f_2,g_1,g_2\in k[t]$ be polynomials
such that 
$$
\gcd(g_1,g_2)=\gcd(f_1,f_2)=1 \in k[t]/k^*.
$$
Let
$$
p(x,y)=g_1(x)f_2(y) -g_2(x)f_1(y)\in k[x,y]
$$
be a polynomial, homogeneous of degree $r\in \N$.
Then 
\begin{align*}
g_i(x)& =a_ix^{r} +b_i,\\
f_i(y)& =c_iy^{r}+d_i,
\end{align*}
for some $a_i,b_i,c_i,d_i \in k$, for $i=1,2$, with 
\begin{align*}
b_1d_2-b_2d_1 & =0,\\
a_1c_2-a_2c_1 & =0.,
\end{align*}
\end{lemm}

\begin{proof}
By homogeneity, $p(0,0)=0$, i.e., 
$$
g_1(0)f_2(0) -g_2(0)f_1(0)=0.
$$
Rescaling, using the symmetry and coprimality of $f_1,f_2$, resp. $g_1,g_2$, 
we may assume that 
$$
\left(\begin{matrix} f_1(0) & f_2(0) \\  g_1(0) & g_2(0) \end{matrix}\right) = 
\left(\begin{matrix} 1 & 1 \\  1 & 1 \end{matrix}\right) \quad \text{or} \quad
\left(\begin{matrix} 1 & 0 \\  1 & 0 \end{matrix}\right).
$$
In the first case, restricting to $x=0$, resp. $y=0$, 
we find
\begin{align*}
g_1(x)-g_2(x)& =ax^r,\\
f_1(y)-f_2(y)& =cy^r,
\end{align*}
for some constants $a,c\in k^*$. 
Solving for $f_2,g_2$ and substituting we obtain
$$
p(x,y)= ax^rf_1(y)-cy^rg_1(x).
$$
In the second case, we have directly 
\begin{align*}
g_1(x)& =ax^r,\\
f_1(y)& =cy^r,
\end{align*}
for some $a,c\in k^*$, and 
$$
p(x,y)= ax^rf_2(y)-cy^rg_2(x).
$$ 
It suffices to apply Lemma~\ref{lemm:red}.
\end{proof}

\begin{prop}
\label{prop:funct-eq}
Let $x,y\in K^*$ be algebraically independent elements.
Fix nonzero integers $r$ and $s$ and consider the equation
\begin{equation}
\label{eqn:rrr}
Ry^r=Sq^s,
\end{equation}
with
$$
R\in k(x/y), \quad p\in k(x), \quad q\in k(y), \quad S\in k(p/q),
$$
where $p\in k(x)$ and $q\in k(y)$ are nonconstant rational functions.
Assume that
\begin{enumerate}
\item[(i)] $x$, $y$, $p$, $q$  are multiplicatively independent;
\item[(ii)] $R,S$ are nonconstant.
\end{enumerate}
Then 
\begin{equation*}
p(x)=\frac{x^{r_1}}{p_{2,1}x^{r_1} + p_2(0)}, 
\quad q(y)=\frac{y^{r_1}}{q_{2,1}y^{r_1} + q_1(0)},
\end{equation*}
or
$$
p(x) = \frac{p_{1,1}x^{r_1} +p_1(0)}{x^{r_1}},
\quad  q(y)=\frac{q_{1,1}y^{r_1} +q_2(0)}{y^{r_1}},
$$
with
$$
r_1\in \N, \quad  p_{1,1},
p_{2,1},p_1(0), p_2(0), q_{1,1},q_{2,1},q_1(0),q_2(0)\in k^*.
$$
We have
$$
Sq^s=\left(\frac{x^{r_1}y^{r_1}}{q_1(0)x^{r_1} -
d_1p_2(0)y^{r_1}}\right)^s
$$
with $d_1=q_{2,1}/p_{2,1}$ and $r=r_1s$ 
in the first case and 
$$
Sq^s=
\left(\frac{p_1(0)y^{r_1} -
d_1q_2(0)x^{r_1}}{x^{r_1}y^{r_1}}\right)^s,
$$
with $d_1=p_{1,1}/q_{1,1}$ and $r=-r_1s$ in the second case.

Conversely, every pair $(p,q)$ 
as above leads to a solution
of \eqref{eqn:rrr}. 
\end{prop} 

\begin{proof}
Equation~\eqref{eqn:rrr} gives, modulo constants, 
\begin{equation}
\label{eqn:ddd}
y^r\prod_{i=0}^I (x/y - c_i)^{n_i} 
=q^s \prod_{j=0}^J (p/q - d_j)^{m_j},
\end{equation}
for pairwise distinct constants $c_i,d_j\in k$, and some $n_i,m_j\in \Z$. 
We assume that $c_0=d_0=0$
and that $c_i,d_j\in k^*$, for $i,j\ge 1$. 
Expanding, we obtain 
\begin{eqnarray*} 
& x^{n_0}y^{r-\sum_{i\geq 0}  n_i}\prod_{i>0} (x - c_iy)^{n_i} =\\ 
& p_1^{m_0} p_2^{-\sum_{j\geq 0} m_j}q_2^{m_0-s} q_1^{s-\sum_{j\geq 0}  m_j}
\prod_{j>0} (p_1q_2 - d_jp_2q_1)^{m_j},
\end{eqnarray*}
where $p=p_1/p_2$ and $q=q_1/q_2$, with $p_1,p_2$ and $q_1,q_2$ coprime
polynomials in $x$, resp. $y$.  
It follows that:

$$
\begin{array}{rl}
(A1) & x^{n_0} =  p_1^{m_0}(x)p_2^{-m_0-\sum_{j> 0}m_j}(x),\\
     & \\
(A2) & y^{r-n_0 - \sum_{i> 0}n_i} = q_2(y)^{m_0 -s} q_1(y)^{s-m_0-\sum_{j> 0}  m_j},\\
     & \\
(A3) &
\prod_{i=1}^I (x - c_iy)^{n_i}= \prod_{j=1}^J (p_1(x)q_2(y) - d_jp_2(x)q_1(y))^{m_j}.
\end{array}
$$

\begin{lemm}
\label{lemm:T2}  
If $n_1 \neq 0$ then the exponents $n_i, m_j$
have the same sign, for all $i,j\ge 1$.
\end{lemm}

\begin{proof} 
Assume otherwise. Collecting terms in (A3) with exponent of the same sign 
we obtain:
\begin{enumerate}
\item[] 
$$
\prod_{i>0, n_i > 0} (x - c_iy)^{n_i} =\prod_{j>0, m_j >0} (p_1q_2 - d_jp_2q_1)^{m_j}, 
$$
\item[] 
$$ \prod_{i>0, n_i < 0} (x - c_iy)^{n_i} =\prod_{j>0, m_j < 0} (p_1q_2 - d_jp_2q_1)^{m_j}$$
\end{enumerate}
Thus there are $a,b\in \N$ such that 
$$
(\prod_{i>0, n_i > 0} (x - c_iy)^{n_i})^a (\prod_{i>0, n_i < 0} (x - c_iy)^{n_i})^b
$$
is a nontrivial rational
function of $x/y$ with trivial divisor at infinity in 
$\P^1\times \P^1$, with standard coordinates $x,y$.
The same holds for
$$
(\prod_{j>0,m_j >0} (p_1q_2 - d_jp_2q_1)^{m_j})^a 
(\prod_{j>0,m_j < 0} (p_1q_2 - d_jp_2 q_1)^{m_j})^b,
$$
a nontrivial rational function of $p/q$.
Thus $k(p/q)\cap k(x/y)\neq k$, which contradicts
the assumption that $p/q$ and $x/y$ are multiplicatively independent.
Indeed, the functions $p/q$ and $x/y$ generate a subgroup of rank 2
in $K^*/k^*$ and hence belong to fields intersecting by constants
only.
\end{proof}

By Lemma~\ref{lemm:T2}, if $\sum_{i> 0} n_i=0$ or $\sum_{j>0}m_j = 0$ 
then $n_i =m_j=0$ for all $i,j \ge 1$. 
By (A1), 
$$
x^{n_0} =  p_1^{m_0}p_2^{-m_0}.
$$ 
By assumption (ii), $R$ is nonconstant. Hence $n_0\neq 0$. 
It follows that $p$ is a power of $x$, contradicting (i).

\

We can now assume 
\begin{equation}
\label{eqn:assume}
\sum_{i > 0}n_i\neq 0, \quad \text{ and } \quad \sum_{i > 0}m_j\neq 0.
\end{equation}
It follows that
\begin{equation*}
\label{eqn:zero}
(m_0, -m_0 -\sum_{j> 0}m_j)\neq (0,0) 
\quad \text {and }\quad (m_0-s,s-m_0-\sum_{j> 0}m_j)\neq (0,0).
\end{equation*}
On the other hand, by (i), combined with (A1) and (A2),
one of the terms in each pair is zero. 
We have the following cases:
\begin{itemize}
\item[(1)] 
$m_0\neq 0$, $m_0=- \sum_{j>0} m _j, \quad m_0=s$ and 
$$
x^{n_0}=p_1^{m_0},  \quad q_1^s=y^{r-n_0-\sum_{i>0}n_i};
$$ 
\item[(2)] 
$m_0=0$, $s=\sum_{j>0}m_j$ and 
$$
x^{n_0}=p_2^{-\sum_{j>0} m_j}= p_2^{-s}, \quad  
q_2^{-s}=y^{r-n_0-\sum_{i>0} n_i}.
$$
\end{itemize}

We turn to (A3), with $J\ge 1$ and $n_i$, $m_j$ replaced by $|n_i|, |m_j|$. 
From (A1) we know that $p_1(x)=x^a$ or $p_2(x)=x^a$, for some $a\in \N$.
Similarly, from (A2) we have  $q_1(y)=y^b$ or $q_2(y)=y^b$, for some $b\in \N$. 
All irreducible components of the divisor of 
$$
f_j:=p_1(x)q_2(y)-d_jp_2(x)q_1(y)
$$
are of the form $x=c_iy$, i.e., these divisors are homogeneous with respect to 
$$
(x,y)\mapsto (\lambda x,\lambda y), \quad \lambda\in k^*.
$$ 
It follows that $f_j$ is homogeneous, of some degree $r_j\in \N$. 
If 
$$
p_1(x)q_2(y)=x^ay^b, \quad \text{ or } \quad p_2(x)q_1(y)=x^ay^b,
$$
then $f_j$ has a nonzero constant term, contradiction. 
Lemma~\ref{lemm:reduction} implies that either
\begin{equation}
\label{eqn:pq}
p_1(x)=x^{r_j} \quad \text{ and } \quad q_1(y)=y^{r_j},
\end{equation}
or 
\begin{equation}
\label{eqn:sym}
p_2(x)=x^{r_j} \quad \text{ and } \quad q_2(y)=y^{r_j}.
\end{equation}
It follows that all $r_j$ are equal, for $j\ge 1$.

The cases are symmetric, and we first consider \eqref{eqn:pq}.
Note that equation~\eqref{eqn:pq} is incompatible with Case $m_0=0$
and equation~\eqref{eqn:sym} with the Case $m\neq 0$. 
By Lemma~\ref{lemm:reduction},   
$$
\begin{array}{rcl}
p_2(x)&= & p_{2,j}x^{r_j} + p_2(0) \\
q_2(y)& =& q_{2,j}y^{r_j} + q_2(0),
\end{array}
$$
with 
\begin{equation}
\label{eqn:j}
p_2(0),q_2(0)\neq 0, \quad \text{ and } \quad q_{2,j} - d_jp_{2,j}=0.
\end{equation}
By assumptions (i), $q_{2,j}$ and $p_{2,j}$ are nonzero. 
The coefficients $d_j$ were distinct, thus there can be at most one
one such equation, i.e., $J=1$.

To summarize, we have the following cases:
\begin{itemize}
\item[(1)]
$m_0\neq 0, m_0=-m_1=s$ and
\begin{equation*}
p(x)=\frac{x^{r_1}}{p_{2,1}x^{r_1} + p_2(0)}, \quad q(y)=\frac{y^{r_1}}{q_{2,1}y^{r_1} + q_1(0)},
\end{equation*}
with coefficients satisfying $q_{2,1} - d_1p_{2,1}=0$,
$$
x^{n_0}=x^{r_1s}, \quad y^{r_1s}=y^{r-n_0-\sum_i n_i},
$$
$$
\prod_{i\ge 1} (x-c_iy)^{n_i}=(q_1(0)x^{r_1}-d_1p_2(0)y^{r_1})^{-s}.
$$
It follows that $I=r_1$ and that $n_i=m_1=-s$, for $i\ge 1$.
We have
$$
c_i=\zeta_{r_1}^i d^{1/r_1},
$$
with $d=-d_1/p_2(0)/q_1(0)$.

This yields $r=n_0=r_1s$.
We can rewrite equation~\eqref{eqn:ddd} as
$$
y^{r_1}\left(\frac{x}{y}\right)^{r_1} 
\prod_{i=1}^{r_1}(\frac{x}{y}-c_i)^{-1}=
\frac{p}{q}
\left(\frac{p}{q}-d_1\right)^{-1}
q,
$$
which is the same as \eqref{eqn:rrr} with $s=1$ and $r=r_1$. 
We have 
\begin{align*}
Sq^s & 
=(q^{-1}-d_1p^{-1})^{-s}\\
& =
\left(\frac{x^{r_1}y^{r_1}}{q_1(0)x^{r_1} -
d_1p_2(0)y^{r_1}}\right)^s.
\end{align*}

\item[(2)]
$m_0=0$, $m_1=s$, and 
$$
p(x) = \frac{p_{1,1}x^{r_1} +p_1(0)}{x^{r_1}},\quad  q(y)=\frac{q_{1,1}y^{r_1} +q_2(0)}{y^{r_1}},
$$
with $p_{1,1}-d_1q_{1,1}=0$,
$$
x^{n_0}=x^{-r_1s}, \quad y^{-r_1s}=y^{r-n_0-\sum_{i>0}n_i}
$$ 
$$
\prod_{i\ge 1} (x-c_iy)^{n_i}=(p_1(0)y^{r_1}-d_1q_2(0)x^{r_1})^s.
$$
We obtain $I=r_1$, $n_i=s$, for $i\ge 1$,  
$n_0=-r_1s=r$,  and
$$
c_i=\zeta_{r_1}^i d^{1/r_1},
$$
with $d=d_1q_2(0)/p_1(0)$.
We can rewrite equation~\eqref{eqn:ddd} as
$$
y^{-r_1}\left(\frac{x}{y}\right)^{-r_1} 
\prod_{i=1}^{r_1}(\frac{x}{y}-c_i)=
\left(\frac{p}{q}-d_1\right)q.
$$
We have
\begin{align*}
Sq^s & = 
(p-d_1q)^s\\
    &=
\left(\frac{p_1(0)y^{r_1} -
d_1q_2(0)x^{r_1}}{x^{r_1}y^{r_1}}\right)^s.
\end{align*}

\end{itemize}

This concludes the proof of Proposition~\ref{prop:funct-eq}.
\end{proof}

\begin{lemm}
\label{lemm:TFL}
Let $x_1, x_2\in K^*$ be algebraically independent elements and 
let $f_i\in \ovl{k(x_i)}$, $i=1,2$.
Assume that $f_1f_2\in \ovl{k(x_1x_2)}$. Then there exists an $a\in \Q$
such that $f_i(x_i)=x_i^a$, in $K^*/k^*$.
\end{lemm}

\begin{proof}
Assume first that $f_i\in k(x_i)$ and write
$$
f_i(x_i)=\prod_{j} (x_i-c_{ij})^{n_{ij}}.
$$
By assumption, 
$$
\prod_{i,j} (x_i-c_{ij})^{n_{ij}} = \prod_{r} (x_1x_2-d_r)^{m_r}.
$$
However, the factors are coprime, unless $c_{ij}=0, d_r=0$, for all $i,j,r$.

Now we consider the general case: $f_i\in \ovl{k(x_i)}$.
We have a diagram of field extensions

\

\centerline{
\xymatrix{
k(x_1x_2) \ar@{-}[r] & \ovl{k(x_1x_2)} \ar@{-}[r] & \ovl{k(x_1,x_2)}   \ar@{-}[d]   \\       
              & \ovl{k(x_1)}\, k(x_2) \ar@{-}[d] \ar@{-}[r]  & \ovl{k(x_1)} \,\,\ovl{k(x_2)} \ar@{-}[d] \\
              & k(x_1,x_2)         \ar@{-}[r]   & k(x_1)\,\ovl{k(x_2)}
}
}

\

The Galois group $\Gal(\,\ovl{k(x_1,x_2)}/k(x_1,x_2))$ preserves $\ovl{k(x_1x_2)}$. 
We have
$$
\Gamma:= \Gal(\,\ovl{k(x_1)}\,\ovl{k(x_2)}/k(x_1,x_2))=\Gamma_1\times \Gamma_2,
$$ 
with $\Gamma_i$ 
acting trivially on $\ovl{k(x_i)}$. 
Put $f_3:=f_1f_2$ and consider the action of $\gamma_1:=(\gamma_1,1)\in \Gamma$ on 
$$
(f_1,f_2,f_3)\mapsto (f_1, \gamma_1(f_2), \gamma_1(f_3)).
$$
It follows that
$$
f_1\gamma_1(f_2) = \gamma_1(f_3),
$$ 
and 
$$
\ovl{k(x_1)}\ni f_2/\gamma_1(f_2)=f_3/\gamma_1(f_3) \in \ovl{k(x_3)}.
$$
Hence each side is in $k$. 
The action of $\gamma_1$ has finite orbit, so that $\gamma_1(f_3)=\zeta_n f_3$
and $\gamma_1(f_2)=\zeta_n' f_2$
for some $n$-th roots of 1. 
Note that $\Gamma$ acts on $f_1,f_2$, and $f_3$ through 
a finite quotient. It follows that for some $m\in \N$, we have 
$f_i^m\in k(x_i)$, for $i=1,2,3$, and we can apply the argument above. 
\end{proof}

\

Let $x,y\in K^*$ 
be algebraically independent over $k$. 
We want to determine the set of 
solutions of the equation
\begin{equation}
\label{eqn:ry}
Ry=Sq,
\end{equation} 
where 
$$
R\in \ovl{k(x/y)},\quad 
q\in \ovl{k(y)},  \quad
p\in \ovl{k(x)},  \quad 
S\in \ovl{k(p/q)}.
$$ 
We assume that
$x,p,y,q$ are  
multiplicatively independent in $K^*/k^*$
and that $S$ and $R$ are nonconstant.  
We will reduce the problem to the one solved in 
Proposition~\ref{prop:funct-eq}.

\begin{lemm} 
\label{lemm:p}
There exists an $n(p)\in \N$ such that
$p^{n(p)}\in \ovl{k(x/y)}\, \ovl{k(y)}$.
\end{lemm}

\begin{proof}
The function $S\in \ovl{k(p/q)}\cap \ovl{k(x/y)}\,\ovl{k(y)}$ 
is nonconstant. The Galois group 
$$
\Gamma:=\Gal(\ovl{k(x,y)}/\ovl{k(x/y)}\,\ovl{k(y)})
$$
acts trivially on $q\in \ovl{k(y)}$ and $S$.
Thus $\ovl{k(p/q)}=\ovl{k(\gamma(p)/q)}$.  
Assume that $\gamma \in \Gamma$ acts nontrivially on $p\in \ovl{k(x)}$. 
It follows that 
$$
\gamma(p)/p \in \ovl{k(p/q)}\cap \ovl{k(x)}=k,
$$
by assumption on these 1-dimensional fields.
Thus $\gamma(p)=\zeta p$, where $\zeta$ is a root of 1.  
Since $\Gamma$ acts on $p$ via a finite quotient
and since each $\gamma\in \Gamma$ acts by multiplication
by a root of 1, $p^{n(p)} \in \ovl{k(x/y)}\, \ovl{k(y)}$, for some $n(p)\in \N$. 
\end{proof}

\begin{lemm}
\label{lemm:pp}
There exists an $N=N(p)\in \N$ such that 
$$
p^{n(p)} \in k(x^{1/N}).
$$
\end{lemm}

\begin{proof} 
The intersection $\ovl{k(x)}\cap \ovl{k(x/y)} \, \ovl{k(y)}$ 
is preserved by action of $\Gamma=\Gamma_{x/y} \times \Gamma_y$. 
Its elements are fixed 
by any lift of
$$
\sigma \,: y \mapsto x/y.
$$ 
to the Galois group $\Gamma$.
All such lifts are obtained by conjugation in 
$\Gamma_{x/y}\times \Gamma_y$.
Hence $(1,\gamma)$ acts as $(\sigma(\gamma),1)$. 
The group homomorphism 
$$
\Gamma_{x/y} \times \Gamma_y \ra \Gamma_x:=\Gal(\ovl{k(x)}/k(x))
$$
has abelian image since $(\gamma_1,1)$ and $(1,\gamma_2)$ 
commute and generate $\Gamma$. 
Every abelian extension of $k(x)$ is described by 
the ramification divisor. 
It remains to observe that 
the only common irreducible divisors of $\ovl{k(y)}$, 
$\ovl{k(x/y)}$ and $\ovl{k(x)}$
are $x=0$ or $x=\infty$.
\end{proof}

\begin{lemm} 
\label{lemm:S}
There exists an $n\in \N$ such that 
$$
S^{n}\in k(x^{1/N},y) \quad \text{ and } \quad q^{n}\in k(y).
$$ 
\end{lemm}

\begin{proof} 
Let
$$
\Gamma'_x \subset \Gamma_x=\Gal(\ovl{k(x)}/k(x^{1/N}))
$$ 
be the subgroup of elements acting trivially
on $k(x^{1/N})$. 
Let 
$$
\gamma= (\gamma_1',1)\in \Gamma_x\times \Gamma_{x/y}, \quad 
\gamma_1' \in \Gamma_x'.
$$
Then 
$$
Ry = Sq=\gamma(S)\gamma(q)\quad \text{ and }\quad  
S/\gamma(S) = \gamma(q)/q.
$$
We also have 
$$
\frac{p/\gamma(q)}{p/q} = q/\gamma(q)
$$ 
with
$$
S \in \ovl{k(p/q)}, \quad 
p/\gamma(q), \,\, \gamma(S) \in \ovl{k(p/\gamma(q))}, \quad 
q/\gamma(q) \in \ovl{k(y)}.
$$
By Lemma~\ref{lemm:TFL},
if we had $\ovl{k(p/q)}\cap \ovl{k(p/\gamma(q))} = k$
then $S=p/q$. However, equation
$Ry=p$ 
and Lemma~\ref{lemm:TFL} imply that $R=x/y$,
contradicting the assumption that $x$ and $p$ are multiplicatively independent. 
Thus we have 
$\ovl{k(p/q)}=\ovl{k(p/\gamma(q))}$.
The equality 
$S/\gamma(S) = (q/\gamma(q))^{-1}$ implies that
both sides are constant.
Hence there exists an $n\in \N$ such that  
$S^{n} \in k(x^{1/N},y)$, and $q^{n}\in k(y)$.
\end{proof}

\begin{lemm}
There exists an $n(R)$ such that 
$R^{n(R)}\in k(\sqrt[N]{x/y})$. 
\end{lemm}

\begin{proof}
We have that 
$$
R^{n}y^{n} =S^{n} q^{n}
$$
with $q^{n}\in k(y)$ and $S^{n}\in k(x^{1/N},y)$.
Thus 
$$
R^{n}\in \ovl{k(x/y)}\cap k(x^{1/N}) k(y).
$$

Applying a nontrivial element $\gamma\in \Gal(\ovl{k(x^{1/N},y)}/k(x^{1/N},y))$
we find that $R^{n}/\gamma(R^{n})\in k^*$, and is thus a root of 1. 
As in the proofs above, we find that 
there is a multiple $n(R)$ of $n$ such that 
$R^{n(R)}\in k(\!\sqrt[N]{x/y}\,)$. 
\end{proof}

We change the coordinates
$$
\tilde{x}:=x^{1/N}, \quad \tilde{y}:=y^{1/N}.
$$

\begin{lemm}
\label{lemm:closure}
There exist 
$$
\tilde{p}\in k(\tilde{x}), \tilde{q}\in k(\tilde{y})
$$
such that 
\begin{equation}
\label{eqn:F}
F:=\ovl{k(p/q)}\cap k(\tilde{x},\tilde{y}) = k(\tilde{p}/\tilde{q}).
\end{equation}
\end{lemm}

\begin{proof}
Every subfield of a rational field is rational. In particular, 
$F=k(\tilde{s})$ for some $\tilde{s}\in k(\tilde{x},\tilde{y})$. 
Since $p\in k(x),q\in k(y)$ they are both in $k(\tilde{x},\tilde{y})$
so that 
$p(x)/q(x)\in F=k(\tilde{s})$. By Lemma~\ref{lemm:fg}, 
$F=k(\tilde{p}/\tilde{q})$, as claimed.  
\end{proof}

\begin{coro}
\label{coro:sss}
There exists an $m\in \N$ such that
$$
S^m\in k(\tilde{p}/\tilde{q}),
$$
with 
$$
\tilde{p}\in k(\tilde{x}) \quad \text{ and }\quad \tilde{q}\in k(\tilde{y}).
$$
Moreover, $\tilde{q}=q^r$, for some $r\in \Q$. 
\end{coro}

\begin{proof}
We apply Lemma~\ref{lemm:TFL}:
since 
$$
\tilde{p} \in k(\tilde{x})\subset \ovl{k(x)}=\ovl{k(p)}, 
\quad 1/\tilde{q}\in \ovl{k(y)}=\ovl{k(1/q)}
$$ 
and 
$$
\tilde{p}/\tilde{q}\in \ovl{k(p/q)},
$$
by \eqref{eqn:F},
$$
\ovl{k(\tilde{p}/\tilde{q})}=\ovl{k(S)}=\ovl{k(p/q)},
$$
we have 
$$
p/q=(\tilde{p}/\tilde{q})^a,
$$
for some $a\in \Q$.  
\end{proof}

We have shown that if $R,S$ satisfy equation \eqref{eqn:ry} then
for all sufficiently divisible $m\in \N$ we have
\begin{equation}
\label{eqn:rnm}
R^m \tilde{y}^{mN} = S^m \tilde{q}^{m/a},
\end{equation}
with 
\begin{equation*}
\label{eqn:srq}
\tilde{S}:=S^m\in k(\tilde{p}/\tilde{q}), 
\quad \tilde{R}:=R^m\in k(\tilde{x}/\tilde{y})
\, \text{ and } \, \tilde{q}:=q^m\in k(y) \subset k(\tilde{y}).
\end{equation*}
Choose a smallest possible $m$ 
such that $s:=m/a\in \Z$ and put
$r=mN$.
Equation~\ref{eqn:rnm} transforms to
$$
\tilde{R}\tilde{y}^{r}=\tilde{S}\tilde{q}^s.
$$
In the proof of 
Proposition~\ref{prop:funct-eq} 
we have shown that $s\mid r$ and that
either
$$
\tilde{R}=
\left(\frac{\tilde{x}}{\tilde{y}}\right)^{r_1s}
\prod_{i=1}^{r_1} 
\left(\frac{\tilde{x}}{\tilde{y}}-c_i\right)^{-s},
\quad
\tilde{S}=\left(\frac{\tilde{p}}{\tilde{q}}\right)^{s}
\left(\frac{\tilde{p}}{\tilde{q}}-d_1\right)^{-s} \tilde{q}^s
$$
with $r_1s=r$ 
or
$$
\tilde{R}=
\left(\frac{\tilde{x}}{\tilde{y}}\right)^{-r_1s}
\prod_{i=1}^{r_1} 
\left(\frac{\tilde{x}}{\tilde{y}}-c_i\right)^{s},
\quad
\tilde{S}=
\left(\frac{\tilde{p}}{\tilde{q}}-d_1\right)^{s} \tilde{q}^s
$$
with $-r_1s=r$.

\

We have obtained that every nonconstant element in 
the intersection
\begin{equation}
\label{eqn:condition}
\ovl{k(x/y)}^*\cdot y \cap \ovl{k(p/q)}^*\cdot q, 
\end{equation}
is of the form
\begin{equation}
\label{eqn:bsN}
\left(\frac{x^by^b}{x^b-\kappa y^b}\right)^s, s\in \N, 
\quad \text{ or } \quad 
\left(\frac{x^b-\kappa'y^b}{x^by^b}\right)^s, -s\in \N,
\end{equation}
with $b=r_1/N$, $N\in \N$, and $\kappa,\kappa'\in k^*$. 
The corresponding solutions, modulo $k^*$, are
\begin{equation*}
p_{\kappa_x,b,m}(x)=\left(\frac{x^{b}}{x^{b} + \kappa_x}\right)^{1/m}, 
\quad q_{\kappa_y,b,m}(y)=\left(\frac{y^{b}}{y^{b} + \kappa_y}\right)^{1/m},
\end{equation*}
with 
$$
\kappa =\kappa_x/\kappa_y
$$
respectively,
$$
p_{\kappa_x,b,m}(x) 
=\left(\frac{x^{b} +\kappa_x'}{x^{b}}\right)^{1/m},\quad  
q_{\kappa_y,b,m}(y)=\left(\frac{y^{b} +\kappa_y'}{y^{b}}\right)^{1/m},
$$
with 
$$
\kappa'=\kappa'_y/\kappa'_x.
$$
By equation \eqref{eqn:ry}, we have (for $s\in \Z$) 
$$
\left(\frac{x^by^b}{x^b-\kappa y^b}\right)^s \cdot y^{-1} \in \overline{k(x/y)}^*
$$
It follows that $bs=1$.

\begin{assu}
\label{assu:xy}
The pair $(x,y)$ satisfies the following condition: if
both $x^b, y^b\in K^*$ then $b\in \Z$.
\end{assu}

This assumption holds e.g., when either
$x,y$ or $xy$ is primitive in $K^*/k^*$.

\begin{lemm}
\label{lemm:bm}
Assume that the pair $(x,y)$ satisfies Assumption~\ref{assu:xy}. 
Fix a solution \eqref{eqn:bsN} of 
Condition~\eqref{eqn:condition}.
Assume that the corresponding 
$p_{\kappa_x,b,m}$ is in $K^*$,
for infinitely many $\kappa_x$, resp. $\kappa_x'$.
Then $b = \pm 1$ and $m=\pm 1$. 
\end{lemm}

\begin{proof}
By the assumption on the pair $(x,y)$ and $K$, 
$$
\frac{x^b}{x^b+\kappa_x} 
$$
is primitive in $K^*/k^*$,  
for infinitely many $\kappa_x$. 
It follows that $m=\pm 1$.  
To deduce that $b=\pm 1$ it suffices 
to recall the definitions: 
on the one hand, $b=r_1/N\in \Z$, with $N \in \N$, 
$r_1\in \N$, and  $r=\pm N$.
Thus, $b=\pm r_1/r\in \Z$. On the other hand, 
$\pm r_1s=r$, with $s\in \N$. 
\end{proof}

After a further substitution $\delta = -b$, we obtain:

\begin{theo}
\label{thm:T} 
Let $x,y\in K^*$ be algebraically independent elements
satisfying Assumption~\ref{assu:xy}.
Let $p\in\ovl{k(x)}^*$, $q\in \ovl{k(y)}^*$ be rational functions such that
$x,y,p,q$ are multiplicatively independent in $K^*/k^*$.
Let $I\in \ovl{k(x/y)}^*\cdot y$ be such that 
there exist infinitely many $p, q\in K^*/k^*$ with  
\begin{equation*}
\label{eqn:condition-2}
I\in \ovl{k(x/y)}^*\cdot y \cap \ovl{k(p/q)}^*\cdot q.
\end{equation*}
Then, modulo $k^*$, 
\begin{equation}
\label{eqn:ixy}
I=I_{\kappa,\delta}(x,y):=
(x^{\delta}-\kappa y^{\delta})^{\delta},
\end{equation}
with $\kappa\in k^*$ and $\delta = \pm 1$. 
The corresponding $p$ and $q$ are given by 
$$
\begin{array}{rclrcl} 
p_{\kappa_x, 1}(x)& = & x+\kappa_x, &  q_{\kappa_y,1}(y)& = & y+\kappa_y\\
p_{\kappa_x,-1}(x)& =& (x^{-1}+\kappa_x)^{-1}, & q_{\kappa_x,-1}(y)& =& (y^{-1}+\kappa_y)^{-1}
\end{array}
$$
with 
$$
\kappa_x/\kappa_y=\kappa.
$$ 
\end{theo}


\section{Reconstruction}
\label{sect:proof}

In this section we prove Theorem~\ref{thm:zero}.
We start with an injective homomorphisms of abelian groups 
$$
\psi_1\,:\, K^*/k^*\ra L^*/l^*.
$$ 
Assume that $z\in K^*$ is primitive
in $K^*/k^*$ and that its image under $\psi_1$ is also
primitive.
Let $x\in K^*$ be an element algebraically independent from $z$ and
put $y=z/x$. By Theorem~\ref{thm:T}, the intersection 
$$
\overline{k(x/y)}^*\cdot y \cap 
\overline{k(p/q)}^* \cdot q  \subset K^*/k^*
$$
with infinitely many corresponding pairs $(p,q)\subset K^*\times K^*$, 
consists of elements $I_{\kappa,\delta}(x,y)$ given in \eqref{eqn:ixy}.
Note that 
$$
I_{\kappa, \delta}(x,y) \neq I_{\kappa', \delta'}(x,y), \quad \text{ for } 
\quad (\kappa,\delta)\neq (\kappa',\delta').
$$
For $\delta=1$, each $I_{\kappa,1}$ determines the infinite sets
$$
\mathfrak l^{\circ}(1,x)=\left\{ 1, x+\kappa_x\right\}_{\kappa_x\in k^*},\quad
\mathfrak l^{\circ}(1,y)=\left\{ 1, y+\kappa_y\right\}_{\kappa_y\in k^*}\quad
$$ 
as the corresponding solutions $(p,q)$. 
The set 
$$
\mathfrak l(1,x):=x\cup \mathfrak l^{\circ}(1,x)\subset \P_k(K)
$$ 
forms a projective line. 
On the other hand, for $\delta=-1$, we get
the set
$$
\mathfrak r(1,x)=\left\{ 1, \frac{1}{x^{-1}+\kappa}\right\}_{\kappa\in k}. 
$$
Note that this set becomes a projective line in $\P_k(K)$, 
after applying the automorphism 
$$
\begin{array}{rcl}
K^*/k^* & \ra &  K^*/k^* \\
f       & \mapsto & f^{-1}. 
\end{array}
$$

We can apply the same arguments to $\psi_1(x), \psi_1(y)=\psi_1(z)/\psi_1(x)$. 
Our assumption that $\psi_1$ maps multiplicative groups
of 1-dimensional subfields of $K$ into 
multiplicative groups of 1-dimensional subfields of $L$  
and Theorem~\ref{thm:T} imply that $\psi_1$ 
maps the projective line 
$\mathfrak l(1,x)\subset \P_k(K)$ to either the projective line 
$\mathfrak l(1,\psi_1(x))\subset \P_l(L)$
or to the set $\mathfrak r(1,\psi_1(x))$. 
Put 
$$
\begin{array}{rl}
\mathcal L & :=
\left\{ x\in K^* \,|\, \psi_1(\mathfrak l(1,x))=\mathfrak l(1,\psi_1(x))
\right\} \\
 &  \\
\mathcal R & := 
\left\{ x\in K^* \,|\, \psi_1(\mathfrak l(1,x))=\mathfrak r(1,\psi_1(x))
\right\}.  
\end{array}
$$ 
Note that these definitions are intrinsic, i.e., they don't depend on 
the choice of $z$. 

By the assumption on $K$, 
both $\mathfrak l(1,\psi_1(x))$ and $\mathfrak r(1,\psi_1(x))$ 
contain infinitely many primitive elements in $L^*/l^*$, 
whose lifts to $L^*$ are algebraically independent from lifts of $\psi_1(z)$. 
We can use these primitive elements as a basis for our constructions
to determine the type of the image 
of $\mathfrak l(1,z')$ for every $z'\in \overline{k(z)}^*\cap K^*$. 
Thus
$$
\mathcal L\cup \mathcal R =K^*/k^*,\quad 
\mathcal L\cap \mathcal R =1 \in K^*/k^*.
$$

\begin{lemm}
\label{lemm:both}
Both sets $\mathcal L$ and $\mathcal R$ are subgroups of $K^*/k^*$. 
In particular, one of these is trivial and the other equal to $K^*/k^*$.
\end{lemm}

\begin{proof}
Assume that $x,y$ are algebraically independent and are both in 
$\mathcal L$. 
We have
$$
\psi_1(I_{\kappa,1}(x,y))=I_{\kappa,1}(\psi_1(x),\psi_1(y)).
$$
Indeed, fix elements 
$$
p(x)=x+\kappa_x\in \mathfrak l(1,x) \quad \text{ and }\quad  
q(y)=y+\kappa_y\in \mathfrak l(1,y)
$$
so that $x,y,p,q$ satisfy the assumptions of Theorem~\ref{thm:T}. 
Solutions of 
$$
R(x/y)y = S(p/q)q
$$
map to solutions of a similar equation in $L$. 
These are exactly
$$
I_{\kappa,1}(\psi_1(x),\psi_1(y))= \psi_1(x)- \lambda \psi_1(y) \in L^*/l^*,
$$
for some $\lambda\in l^*$. 
This implies that   
$$
\psi_1(x/y -\kappa) = \psi_1(x/y)- \lambda \in L^*/l^*,
$$
i.e., 
$x/y\in \mathcal L$.

Now we show that if $x\in \mathcal L$ then every 
$x'\in \overline{k(x)}^*/k^*\cap  K^*/k^*$
is also in $\mathcal L$. First of all, 
$1/x\in  \mathcal L$. 
Next, elements in the ring $k[x]$, modulo $k^*$,  
can be written as products of 
linear terms $x+\kappa_i$. 
Hence 
$$
\psi_1(k[x]/k^*)\subset l[\psi_1(x)]/l^*.
$$
Let $f$ be integral over $k[x]$ and let
$$
f^n +\ldots + a_0(x)\in k[x]
$$
be the minimal polynomial for $f$,
where $a_0(x)\notin k$. Replacing $f$ by $f+\kappa$, if necessary, 
we may assume that $f$ is not a unit in the ring $\overline{k[x]}$. 
Then $f\notin \mathcal R$, since otherwise we would have
$a_0(x)\in \mathcal R$, contradiction. 
Finally, any element of $\overline{k(x)}^*$ is contained in
the integral closure of some $k[1/g(x)]$, with $g(x)\in \overline{k[x]}$.

The same argument applies to 
$\mathcal R$, once we composed with $\psi_1^{-1}$, to show
that both $\mathcal L$ and $\mathcal R$ are subgroups of $K^*/k^*$.
An abelian group cannot be a union of two subgroups intersecting only in
the identity. Thus either $\mathcal L$ or $\mathcal R$ has to be trivial. 
\end{proof}

The set $\P(K)=K^*/k^*$ carries two compatible structures: of 
an abelian group and a projective space, 
with projective subspaces preserved by
the multiplication. 
The projective structure on the multiplicative group 
$\P(K)$ encodes the field structure:

\begin{prop}\cite[Section 3]{bt}
\label{prop:proj}
Let $K/k$ and $L/l$ be geometric fields 
over $k$, resp. $l$, of transcendence of degree $\ge 2$.  
Assume that 
$\psi_1\,:\, K^*/k^*\ra L^*/l^*$
maps lines in $\P(K)$ into lines in $\P(L)$. 
Then $\psi_1$ is a morphism of projective structures,
$\psi_1(\P(K))$ is a projective subspace in $\P(L)$,
and there exist a subfield $L'\subset L$ and 
an isomorphism of fields 
$$
\psi\,:\, K\ra L',
$$ 
which compatible with $\psi_1$.  
\end{prop}

Lemma~\ref{lemm:both} shows that either $\psi_1$ or $\psi_1^{-1}$
satisfies the conditions of Proposition~\ref{prop:proj}. 
This proves Theorem~\ref{thm:zero}.

\

\section{Milnor {\rm K}-groups}
\label{sect:prelim}

Let $K=k(X)$ be a function field of an algebraic variety $X$ over
an algebraically closed field $k$. 
In this section we characterize intrinsically infinitely divisible 
elements in ${\rm K}_1^M(K)$ and  ${\rm K}_2^M(K)$.
For $f\in K^*$ put 
\begin{equation}
\label{eqn:ker22}
\Ker_2(f):=\{\, g\in K^*/k^* = \bar{\rK}^M_1(K)\,\mid \, 
(f,g)= 0 \in \bar{\rK}_2^M(K) \,\}.
\end{equation} 

\begin{lemm}
\label{lemm:k1}
An element $f\in K^*={\rm K}_1^M(K)$ is infinitely divisible if and only if
$f\in k^*$.  
In particular, 
\begin{equation}
\label{mult}
\bar{\rK}^M_1(K)=K^*/k^*.
\end{equation}
\end{lemm}

\begin{proof}
First of all, every element in $k^*$ is infinitely divisible, 
since $k$ is algebraically closed. 
We have an exact sequence
$$
0\to k^*\to K^*\to \Div(X).
$$ 
The elements of $\Div(X)$ are not infinitely divisible. 
Hence every infinitely divisible element of $K^*$ is in $k^*$. 
\end{proof}

\begin{lemm}
\label{lemm:divis}  
Given a nonconstant $f_1\in K^*/k^*$, we have 
$$
\Ker_2(f_1)=E^*/k^*,
$$ 
where $E=\overline{k(f_1)}\cap K$. 
\end{lemm}

\begin{proof} 
Let $X$ be a normal projective model of $K$. 
Assume first that $f_1,f_2\in K\setminus k$ 
lie in a 1-dimensional subfield $E\subset K$ 
that contains $k$ and is normally closed in $K$.
Such a field $E$ defines a rational map $\pi\,:\, X\ra C$, 
where $C$ is a projective model of $E$. 

By the Merkurjev--Suslin theorem \cite{MS},
for any field $F$ containing $n$-th roots of unity one has
$$
{\rm Br}(F)[n]= {\rm K}^M_2(F))/({\rm K}^M_2(F))^n,
$$ 
where ${\rm Br}(F)[n]$ is the $n$-torsion subgroup of
the Brauer group ${\rm Br}(F)$.
On the other hand, by Tsen's theorem, ${\rm Br}(E)= 0$, since
$E=k(C)$, and $k$ is algebraically closed.  
Thus the symbol 
$(f_1,f_2)$ is infinitely divisible in ${\rm K}_2^M(E)$ and hence
in ${\rm K}_2^M(K)$.

Conversely, assume that the symbol $(f_1,f_2)$ is infinitely divisible 
in  ${\rm K}_2^M(K)$ and that the field $k(f_1,f_2)$ has transcendence degree two.
Choosing an appropriate model of $X$, we may assume that
the functions $f_i$ define surjective morphisms $\pi_i\,:\, X \ra \P^1_i=\P^1$, 
and hence a proper surjective map $\pi\,:\, X\ra \P^1_1\times \P^1_2$.  

For any irreducible divisor $D\subset X$ the restriction 
of the symbol $(f_1,f_2)$ to $D$ is well-defined, as an element 
of ${\rm K}_1^M(k(D))$. It has to be infinitely divisible in  
${\rm K}_1^M(k(D))$, for each $D$.

For $j=1,2$, consider the divisors 
${\rm div}(f_j)=\sum n_{ij} D_{ij}$, where $D_{ij}$ 
are irreducible.
Let  $D_{11}$ be a component surjecting onto $\P^1_1\times 0$. 
The restriction of $f_2$ to $D_{11}$ is nonconstant. 
Thus $D_{11}$ is not a component in the divisor of $f_2$ and 
the residue 
$$
\varrho(f_1,f_2)\in {\rm K}_1^M(k(D_{11})^*)=(f_2|D_{11})^{n_{11}}\notin k^*.
$$
It remains to apply Lemma~\ref{lemm:k1} to conclude that the 
residue and hence the symbol are not divisible. This 
contradicts the assumption that $k(f_1,f_2)$ has transcendence 
degree two. 
\end{proof}

\begin{coro}
\label{coro:two}
Let $K$ and $L$ be function fields over $k$. 
Any group homomorphism 
$$
\psi_1\,:\,  \bar{{\rm K}}_1^M(K)\ra \bar{{\rm K}}_1^M(L)
$$
satisfying the assumptions of Theorem~\ref{thm:main}
maps multiplicative subgroups of 
normally closed one-dimensional subfields
of $K$ to multiplicative subgroups of 
one-dimensional subfields of $L$. 
\end{coro}

We now prove Theorem~\ref{thm:main}.

\

{\em Step 1.}
For each normally closed one-dimensional subfield 
$E\subset K$ there exists a one-dimensional
subfield $\tilde{E}\subset L$ such that 
$$
\psi_1(E^*/k^*)\subset \tilde{E}^*/l^*
$$
Indeed, Lemma~\ref{lemm:divis} identifies
multiplicative groups of 1-dimensional normally closed subfields in $K$:
For $x\in K^*\setminus k^*$ the group $\overline{k(x)}^*\subset K^*$ is 
the set of all $y\in K^*/k^*$ such that the symbol 
$(x,y)\in \bar{\rK}_2^M(K)$ is zero. 

\

{\em Step 2.}
There exists an $r\in \N$ such that 
$\psi_1^{1/r}(K^*/k^*)$ contains a primitive element of $L^*/l^*$. 
Note that $L^*/l^*$ is torsion-free. 
For $f,g\in K^*/k^*$ assume that 
$\psi_1(f), \psi_1(g)$ are $n_f$, resp. $n_g$, powers of primitive, 
multiplicatively independent  
elements in $L^*/l^*$. 
Let $M:=\langle \psi_1(f), \psi_1(g)\rangle$ and let ${\rm Prim}(M)$ 
be its primitivization. Then ${\rm Prim}(M)/M=\Z/n\oplus \Z/m$, 
with $n\mid m$, i.e., $n=\gcd(n_f,n_g)$.
Thus, we can take $r$ to be is the 
smallest nontrivial power of an element  in
$\psi_1(K^*/k^*)\subset L^*/l^*$.

\

{\em Step 3.}
By Theorem~\ref{thm:zero} either $\psi_1^{1/r}$ or $\psi_1^{-1/r}$ extends to a 
homomorphism of fields.

\bibliographystyle{smfplain}
\bibliography{milnor}
\end{document}